\documentclass[12pt,leqno]{article}

\usepackage{amssymb, graphicx, amscd, amsmath, amssymb, amsthm, array, fancyvrb, relsize, paralist}

\newtheorem{thm}{Theorem}[section]
\newtheorem{lem}[thm]{Lemma}

\numberwithin{equation}{section}

\renewcommand\Re{\operatorname{Re}}

\def\Q{\mathbb{Q}}
\def\F{\mathbb{F}_q}
\def\FF{\mathbb{F}_{q^2}}

\def\3F2{{}_3\hspace{-1pt}F_2}
\def\2F1{{}_2\hspace{-1pt}F_1}
\def\h2F1{{}_2\hspace{-1pt}\widehat{F}_1}
\def\f{\F^*}
\def\ff{\FF^*}
\def\e{\varepsilon}
\def\co{\chi_1}

\def\oc{\overline{\chi}}

\def\no{\nu_1}
\def\ono{\overline{\no}}

\def\onu{\overline{\nu}}

\def\omu{\overline{\mu}}
\def\oA{\overline{A}}
\def\omf{\overline{M}_4}
\def\ome{\overline{M}_8}
\def\oB{\overline{B}}
\def\oC{\overline{C}}
\def\oD{\overline{D}}

\def\oJ{\overline{J}}

\def\oz{\overline{z}}

\def\opi{\overline{\pi}}

\def\l({\left(}
\def\r){\right)}
\def\bar{\begin{array}{r|}}
\def\ear{\end{array}}

\def\Re{\mathrm{\ Re \,}}

\title{SOME MIXED CHARACTER SUM IDENTITIES OF KATZ II}

\author{\\ \\ Ron Evans\\
Department of Mathematics\\
University of California at San Diego\\
La Jolla, CA  92093-0112\\
revans@ucsd.edu
\\ \\
and
\\ \\
John Greene\\
Department of Mathematics and Statistics\\
University of Minnesota--Duluth\\
Duluth, MN  55812\\
jgreene@d.umn.edu
}

\medskip

\date{December 2016}

\begin{document}
\maketitle



\newpage

\begin{abstract}
A conjecture connected with quantum physics led N. Katz to discover
some amazing mixed character sum identities over a field of $q$
elements, where $q$ is a power of a prime $p >3$. 
His proof required deep algebro-geometric techniques,
and he expressed interest in finding a more straightforward direct proof.
The first author recently 
gave such a proof of his identities when $q \equiv 1 \pmod 4$,
and this paper provides such a proof for 
the remaining case $q \equiv 3 \pmod 4$.
Our proofs are valid for all characteristics $p>2$.
Along the way we prove some elegant new character sum identities.

\end{abstract}

\maketitle

\section{Introduction}
Let $\F$ be a field of $q$ elements, 
where $q$ is a power of an odd prime $p$. 
Throughout this paper, $A$, $B$, $C$, $D$, 
$\chi$, $\lambda$, $\nu$, $\mu$, $\e$, $\phi$
denote complex multiplicative characters on $\f$, 
extended to map 0 to 0. 
Here $\e$ and $\phi$ always denote the trivial and quadratic characters, 
respectively.  
Define $\delta(A)$ to be 1 or 0 according as $A$ is trivial or not,
and let $\delta(j,k)$ denote the Kronecker delta for $j,k \in \F$.

Much of this paper deals with the extension field $\FF$ of $\F$.
Let $M_4$ denote a fixed quartic character on $\FF$ and let
$M_8$ denote a fixed octic character on $\FF$ such that
$M_8^2 = M_4$.

Define the additive character  $\psi$ on $\F$ by
\begin{equation*}
\psi(y) := 
\exp \Bigg( \frac{2 \pi i}{p} \Big( y^p + y^{p^2} + \dots + y^q \Big) \Bigg),
\quad y \in \F.
\end{equation*}
The corresponding additive character on $\FF$ will be denoted by $\psi_2$.

Recall the definitions of the Gauss and Jacobi sums over $\F$:
\begin{equation*}
G(A) = \sum_{y \in \F} A(y) \psi(y), \quad
J(A, B) = \sum_{y \in \F} A(y) B(1-y).
\end{equation*}
These sums have the familiar properties 
\[
G(\e) = -1, \quad J(\e,\e) = q-2,
\]
and for nontrivial $A$, 
\[
G(A) G(\oA) = A(-1) q, \quad J(A, \oA) = -A(-1), 
\quad J(\e, A)=-1. 
\]
Gauss and Jacobi sums are related by \cite[p. 59]{BEW}
\begin{equation*}
J(A,B) = \frac{G(A) G(B)}{G(AB)}, \quad \text{if } AB \neq \e
\end{equation*}
and 
\begin{equation*}
J(A,\oC) = \frac{A(-1)G(A) G(\oA C)}{G(C)}=A(-1)J(A,\oA C), 
\quad \text{if } C \neq \e.
\end{equation*}
The Hasse--Davenport product relation \cite[p. 351]{BEW} yields
\begin{equation}\label{eq:1.1}
A(4) G(A) G(A \phi) = G(A^2) G(\phi).
\end{equation}

As in \cite[p. 82]{TAMS}, 
define the hypergeometric $\2F1$ function over $\F$ by
\begin{equation}\label{eq:1.2}
\2F1 \l( \bar A,B \\ C \ear \ x \r)
=\frac{\e (x)}{q}
\sum_{y \in \F} B(y) \oB C(y-1)
\oA(1-xy), \quad x \in \F.
\end{equation}

For $j, k \in \F$ and $a \in \f$, 
Katz \cite[p. 224]{Katz} defined the mixed exponential sums
\begin{equation}\label{eq:1.3}
\begin{split}
P(j,k): &= \delta(j,k) + \phi(-1)\delta(j,-k) + \\
&  \frac{1}{G(\phi)}\sum_{x \in \f}
\phi(a/x - x)\psi(x(j+k)^2 + (a/x)(j-k)^2).
\end{split}
\end{equation}
Note that 
\begin{equation}\label{eq:1.4}
P(j,k)=P(k,j), \quad
P(-j,k) = \phi(-1)P(j,k).
\end{equation}
Katz proved an equidistribution conjecture of Wootters 
\cite[p. 226]{Katz}, \cite{ASSW}
connected with quantum physics
by constructing explicit character sums $V(j)$ 
 \cite[pp. 226--229]{Katz}) for which
the identities
\begin{equation}\label{eq:1.5}
P(j,k) = V(j)V(k)
\end{equation}
hold for all $j,k \in \F$.
(The $q$-dimensional vector 
$(V(j))_{j \in \F} $
is a minimum uncertainty state, as described by Sussman and
Wootters \cite{SW}.)
Katz's proof \cite[Theorem 10.2]{Katz} of the identities
\eqref{eq:1.5} required the characteristic $p$ to exceed 3,
in order to guarantee that various sheaves of ranks 2, 3, and 4
have geometric and arithmetic monodromy groups
which are SL(2), SO(3), and SO(4), respectively.

As Katz indicated in \cite[p. 223]{Katz}, his proof of
\eqref{eq:1.5} is quite complex, invoking the theory of
Kloosterman sheaves and their rigidity properties, as well as results
of Deligne \cite{Del} and Beilinson, Bernstein, Deligne \cite{BBD}.  
Katz \cite[p. 223]{Katz} wrote,
``It would be interesting to find direct proofs of these identities."

The goal of this paper is to respond to
Katz's challenge by giving a direct proof of \eqref{eq:1.5}
( a ``character sum proof"
not involving algebraic geometry).
This has the benefit of making
the demonstration of his useful identities 
accessible to a wider audience of mathematicians and physicists.
Since a direct proof for $q \equiv 1 \pmod 4$ has been given in
\cite{Evans}, we will assume from here on that $q \equiv 3 \pmod 4$.

A big advantage of our proof is that it works for
all odd characteristics $p$, including $p=3$.
As a bonus, we obtain some elegant new double character sum
evaluations in \eqref{eq:5.11}--\eqref{eq:5.14}.

Our method of proof is to show
(in Section 6) 
that the double Mellin transforms of
both sides of \eqref{eq:1.5} are equal.  
The Mellin transforms of the left and right sides of \eqref{eq:1.5}
are given in Theorems 3.1 and 5.1, respectively.
A key feature of our proof
is a formula (Theorem 4.1) relating a norm-restricted Jacobi sum over $\FF$
to a
hypergeometric $\2F1$  character sum over $\F$.
Theorem 4.1 will be applied to prove Theorem 5.3, 
an identity for a weighted sum of hypergeometric $\2F1$ 
character sums.   Theorem 5.3 is crucial
for our proof of \eqref{eq:1.5} in Section 6.

Hypergeometric character sums over finite fields have had a variety of
applications in number theory.   For some recent examples, see
\cite{BarKal}, \cite{BarSai}, \cite{ElOno}, \cite{FLRST},
\cite{LinTu}, \cite{McPap}, \cite{Salerno}.

Since $q \equiv 3 \pmod 4$, we have $\phi(-1)=-1$, and 
every element $z \in \FF$ has the form
\[
z= x + iy, \quad   x,y \in \F,
\]
where $i$ is a fixed primitive fourth root of unity in $\FF$.
Write $\oz = x - iy$ and note that $\oz = z^q$.
The restriction of $M_8$ to $\F$ equals $\e$ or $\phi$ according
as $q$ is congruent to 7 or 3 mod 8.  In particular,
\begin{equation}\label{eq:1.6}
M_8(-1) = \phi(2).
\end{equation}
For a character $C$ on $\F$, we let $CN$ denote the character on $\FF$
obtained by composing $C$ with the norm map $N$ on $\FF$ defined by
\[
Nz = z\oz \in \F.
\]
Given a character $B$ on $\F$, $BCN$ is to be interpreted as the character
$(BC)N$, i.e., $BN CN$.

For the same $a$ as in \eqref{eq:1.3}, define
\[
\tau = -\sqrt{qM_8(-a)},
\]
where the choice of square root is fixed.
Katz defined the  sums $V(j)$ to be the following norm-restricted
Gauss sums:
\begin{equation}\label{eq:1.7}
V(j): = \tau^{-1}\phi(j)
\sum_{\substack{z \in \FF \\  Nz=a}} M_8(z)\psi_2(j^2 z), \quad j \in \F.
\end{equation}
Note that 
\begin{equation}\label{eq:1.8}
V(-j) = -V(j).
\end{equation}

\section{Mellin transform of the sums $V(j)$}

This section begins  with some results related to
Gauss sums over $\FF$ that will be used in this paper.
We use the notation $G_2$ and $J_2$ for Gauss and Jacobi sums
over $\FF$, in order to distinguish them from the Gauss and Jacobi
sums $G$ and $J$ over $\F$.
For any character $\beta$ on $\FF$, we have
\begin{equation}\label{eq:2.1}
G_2(\beta) = G_2(\beta^q);
\end{equation}
for example, for a character $C$ on $\F$,
$G_2(CN M_8)$ equals $G_2(CN \ome)$ or $G_2(CN M_8^3)$
according as $q$ is congruent to 7 or 3 mod 8.
The Hasse-Davenport 
theorem on lifted Gauss sums \cite[Theorem 11.5.2]{BEW} gives
\begin{equation}\label{eq:2.2}
G_2(CN) = - G(C)^2.
\end{equation}
From \cite[(4.10)]{Hiro},
\begin{equation}\label{eq:2.3}
G_2(CN M_4) = G_2(CN \omf) = -\oC^2\phi(2)G(C^2\phi)G(\phi).
\end{equation}
For any character $\beta$ on $\FF$, define
\begin{equation}\label{eq:2.4}
E(\beta): = \sum_{y \in \F} \beta(1+iy).
\end{equation}
It is easily seen that
\begin{equation}\label{eq:2.5}
E(\beta) = \beta(2) E_2(\beta),
\end{equation}
where $E_2(\beta)$ is the Eisenstein sum
\begin{equation}\label{eq:2.6}
E_2(\beta): = \sum_{\substack{z \in \FF \\  z+z^q =1}} \beta(z).
\end{equation}
Let $\beta^*$ denote the restriction of $\beta$ to $\F$.
Applying  \cite[Theorem 12.1.1]{BEW} with $q$ in place of $p$,
we can express $E_2(\beta)$ in terms of Gauss sums when
$\beta$ is nontrivial, as follows:
\begin{equation}\label{eq:2.7}
E_2(\beta) = \begin{cases} 
G_2(\beta)/G(\beta^*) &\mbox{if } \beta^* \ne \e \\ 
-G_2(\beta)/q & \mbox{if } \beta^* = \e . \end{cases}
\end{equation}

For any character $\chi$ on $\F$, define the Mellin transform
\begin{equation}\label{eq:2.8}
S(\chi): = \sum_{j \in \f} \chi(j)V(j).
\end{equation}
In the case that $\chi$ is odd, we may write
$\chi = \phi \lambda^2$
for some character $\lambda$ on $\F$.
In that case, we may assume without loss of generality
that $\lambda$ is even, otherwise replace $\lambda$ by $\phi \lambda$.
In summary, when $\chi$ is odd, 
\begin{equation}\label{eq:2.9}
\chi = \phi \lambda^2 = \phi \nu^4, \quad \lambda = \nu^2
\end{equation}
for some character $\nu$ on $\F$.

The next theorem gives an evaluation of $S(\chi)$ in terms of Gauss sums.

\begin{thm}
If $\chi$ is even, then $S(\chi)=0$.  If $\chi$ is odd 
(so that \eqref{eq:2.9} holds), then
\begin{equation}\label{eq:2.10}
S(\chi)= 
\onu(a)\tau^{-1}G_2(\nu N M_8)+\phi\onu(a)\tau^{-1}G_2(\nu N M_8^5).
\end{equation}
\end{thm}

\begin{proof}
If $\chi$ is even, then $S(\chi)$ vanishes by (1.8) and (2.8).
Now assume that $\chi$ is odd, so that $\chi = \phi \nu^4$.  Then
\begin{equation*}
S(\chi)=\frac{\tau^{-1}}{q-1}\sum_{z \in \FF} \sum_{j \in \f}
M_8(z)\psi_2(z j^2) \nu^4(j) \sum_{C} C(N(z)/a).
\end{equation*}
Replace $z$ by $z/j^2$ to get
\begin{equation*}
S(\chi)=\frac{\tau^{-1}}{q-1}\sum_{C} \oC(a)\sum_{z \in \FF} 
M_8(z)\psi_2(z)CN(z)\sum_{j \in \f} \nu^4 \oC^4(j).
\end{equation*}
The sum on $j$ on the right equals $q-1$ when $C \in \{\nu, \phi \nu\}$
and it equals $0$ otherwise.   Since $\phi N = M_8^4$, the result
now follows from the definition of $G_2$.
\end{proof}

\section{Double Mellin transform of $V(j)V(k)$}

For characters $\chi_1, \chi_2$, define the double Mellin transform
\begin{equation}\label{eq:3.1}
S=S(\chi_1, \chi_2): = \sum_{j, k \in \f} \chi_1(j)\chi_2(k)V(j)V(k).
\end{equation}
As in \eqref{eq:2.9},
when $\chi_1$ and $\chi_2$ are both odd,
\begin{equation}\label{eq:3.2}
\chi_i = \phi \lambda_i^2 = \phi \nu_i^4, \quad \lambda_i = \nu_i^2,
\quad i=1,2,
\end{equation}
for some characters $\nu_1$, $\nu_2$ on $\F$.
In this case, write
\begin{equation}\label{eq:3.3}
\mu = \nu_1 \nu_2.
\end{equation}

The following theorem evaluates $S$ in terms of Gauss and Jacobi
sums.

\begin{thm}
If $\chi_1$ or $\chi_2$ is even, then $S=0$.
If $\chi_1$ and $\chi_2$ are both odd
(so that \eqref{eq:3.2} and \eqref{eq:3.3} hold), then
\begin{equation}\label{eq:3.4}
S=\sum_{i=0}^{1} \phi^i \omu(a)\frac{q }{G_2(\phi^i \omu N)}
\{J_2(\nu_1 N M_8, \phi^i \omu N) + J_2(\nu_1 N M_8^5, \phi^i \omu N)\}. 
\end{equation}
\end{thm}

\begin{proof}
By \eqref{eq:3.1}, $S=S(\chi_1)S(\chi_2)$. By Theorem 2.1, 
$S=0$ when $\chi_1$ or $\chi_2$ is even.  Thus assume that
$\chi_1$ and $\chi_2$ are both odd.  Then Theorem 2.1 yields
\begin{equation}\label{eq:3.5}
\begin{split}
S=&\sum_{i=0}^{1} \phi^i \omu(a)\frac{M_8(-a) }{q} \ \times \\
&\times \ \{G_2(\nu_1 N M_8)G_2( \phi^i \mu \ono N M_8) + 
G_2(\nu_1 N M_8^5)G_2( \phi^i \mu \ono N M_8^5)\}.
\end{split}
\end{equation}
A straightforward computation
with the aid of \eqref{eq:2.1}
shows that \eqref{eq:3.5} is equivalent to  \eqref{eq:3.4}.
The computation is facilitated by noting that
$M_8(-a)$ equals $1$ or $-\phi(a)$ according as
$q$ is congruent to 7 or 3 mod 8, 
so that the bracketed expression for $i=0$ in \eqref{eq:3.4}
is to be compared to that for  $i=1$ in \eqref{eq:3.5}
when $q$ is congruent to 3 mod 8.
\end{proof}

\section{Identity for a norm-restricted Jacobi sum in terms of a $\2F1$}

Let $D$ be a character on $\F$. Define the norm-restricted Jacobi sums
\begin{equation}\label{eq:4.1}
R(D,j):=\sum_{\substack{z \in \FF \\ N(z)=j^4}}
M_8(z)\oD N(1-z), \quad j \in \f.
\end{equation}

The next theorem provides a formula expressing $R(D,j)$
in terms of a $\2F1$ hypergeometric character sum.

\begin{thm}
For $j = \pm 1$, 
\begin{equation}\label{eq:4.2}
R(D,j) = -\oD(4)J(\phi D^2, \phi).
\end{equation}
For all other $j \in \f$,
\begin{equation}\label{eq:4.3}
R(D,j)=-\phi(j) q \oD^4(j-1)
\2F1 \l( \bar D,D^2\phi \\ D\phi \ear \
-\left(\frac{j+1}{j-1}\right)^2  \r).
\end{equation}
\end{thm}

\begin{proof}
Replace $z$ in \eqref{eq:4.1} by $-zj^2$.
By \eqref{eq:1.6}, we obtain
\[
R(D,j)=\phi(2)\sum_{\substack{z \in \FF \\ N(z)=1}}
M_8(z)\oD N(1+zj^2).
\]
Each $z$ in the sum must be a square, since $N(z)$ is a square in $\F$. Thus
\[
R(D,j)=\frac{\phi(2)}{2}\sum_{\substack{z \in \FF \\ N(z)=1}}
M_4(z)\oD N(1+z^2j^2).
\]
Writing $z=x+iy$, we have
\[
R(D,j)=\frac{\phi(2)}{2}\sum_{x^2+y^2=1}
M_4(x+iy)\oD ((1-j^2)^2 + 4 j^2 x^2),
\]
where it is understood that the sum is over all $x,y \in \F$
for which $x^2 + y^2 =1$.
Thus, since $M_4(\pm i)=M_8(-1)=\phi(2)$,
\begin{equation}\label{eq:4.4}
R(D,j) = \oD^2(1-j^2) + Q(D,j),
\end{equation}
where
\[
Q(D,j)=\frac{\phi(2)}{2}\sum_{\substack{x^2+y^2=1 \\ x \ne 0}}
M_4(x+iy)\oD ((1-j^2)^2 + 4 j^2 x^2).
\]
Replacing $y$ by $yx$, we have
\[
Q(D,j)=\frac{\phi(2)}{2}\sum_{\substack{1+y^2=x^{-2} \\ x \ne 0}}
M_4(1+iy)\oD ((1-j^2)^2 + 4 j^2 /(1+ y^2)).
\]
Since $\omf = \phi N M_4$, this yields
\begin{equation}\label{eq:4.5}
Q(D,j)=\frac{\phi(2)}{2}\sum_{y \in \F}
\{M_4(1+iy)+\omf(1+iy)\} \oD ((1-j^2)^2 + 4 j^2 /(1+ y^2)).
\end{equation}

First consider the case where $j = \pm 1$.
By \eqref{eq:4.4} and \eqref{eq:4.5},
\[
R(D,j)=Q(D,j) = \frac{\oD(4) \phi(2)}{2}\sum_{y \in \F}
\{DN M_4(1+iy)+DN \omf(1+iy)\}.
\]
By \eqref{eq:2.4} and \eqref{eq:2.5},
\begin{equation}\label{eq:4.6}
R(D,j)=\frac{\phi(2)}{2}\{E_2(DN M_4) + E_2(DN \omf)\}.
\end{equation}
The restriction of $DN M_4$ to $\F$ is $D^2$.   Thus by \eqref{eq:2.7},
\[
R(D,j)=\frac{\phi(2)}{2 G(D^2)}\{G_2(DN M_4) + G_2(DN \omf)\},
\]
if $D^2$ is nontrivial, and
\[
R(D,j)=\frac{\phi(2)}{-2 q}\{G_2(DN M_4) + G_2(DN \omf)\},
\]
if $D^2$ is trivial.
By \eqref{eq:2.3},
\[
G_2(DN M_4) = G_2(DN \omf) = -\oD(4) \phi(2)G(D^2 \phi)G(\phi).
\]
Consequently,
\begin{equation}\label{eq:4.7}
R(D,j) = -\oD(4)J(\phi D^2, \phi)
\end{equation}
for every $D$, which completes the proof when $j = \pm 1$.
Thus assume for the remainder of this proof that $j^2 \ne 1$.

By \eqref{eq:4.5}, $Q(D,j)$ equals
\[
\frac{\oD^2(1-j^2)}{2\phi(2)}\sum_{y \in \F}
\{M_4(1+iy)+\omf(1+iy)\} 
\oD \left(1 + \frac{4 j^2}{(1+ y^2)(1-j^2)^2}\right).
\]
By the ``binomial theorem" \cite[(2.10)]{TAMS}, the rightmost
factor above equals
\[
\frac{q}{q-1}\sum_{\chi}
\begin{pmatrix} D \chi \\ \chi \end{pmatrix}
\chi\left(\frac{-4 j^2}{(1+ y^2)(1-j^2)^2}\right),
\]
where the ``binomial coefficient" over $\F$ is defined by 
\cite[p. 80]{TAMS}
\begin{equation*}
\begin{pmatrix} A \\ B \end{pmatrix}
= \frac{B(-1)}{q} J(A, \overline{B}).
\end{equation*}
Replacing $\chi$ with $\oc$ and observing that \cite[p. 80]{TAMS}
\[
\begin{pmatrix} D \oc \\ \oc \end{pmatrix}=
D(-1)\begin{pmatrix}  \chi \\ \oD \chi \end{pmatrix},
\]
we see that
\[
Q(D,j)=\frac{\oD^2(1-j^2)D(-1)\phi(2)q}{2(q-1)}\sum_{\chi}
\begin{pmatrix}  \chi \\ \oD \chi \end{pmatrix}
\chi\left(\frac{-(1-j^2)^2}{4j^2}\right)\kappa(\chi),
\]
where
\[
\kappa(\chi):=\sum_{y \in \F}
\{\chi N M_4(1+iy) + \chi N \omf(1+iy)\}.
\]
By \eqref{eq:2.4} and \eqref{eq:2.5},
\[
\kappa(\chi)=\chi(4)\{E_2(\chi N M_4) + E_2(\chi N \omf)\}.
\]
Comparing \eqref{eq:4.6} and \eqref{eq:4.7},
we see that
\[
\kappa(\chi)= -2\phi(2)J(\phi \chi^2, \phi)=
2q\phi(2)\oc(4)
\begin{pmatrix} \phi \chi^2 \\ \chi \end{pmatrix},
\]
where the last equality follows from the Hasse-Davenport
relation \eqref{eq:1.1}.
Consequently,
\[
Q(D,j)=\frac{\oD^2(1-j^2)D(-1)q^2}{q-1}\sum_{\chi}
\begin{pmatrix}  \chi \\ \oD \chi \end{pmatrix}
\begin{pmatrix} \phi \chi^2 \\ \chi \end{pmatrix}
\chi\left(\frac{-(1-j^2)^2}{16j^2}\right).
\]
Replace $\chi$ by $D \chi$ to get
\[
Q(D,j)=\frac{\oD(16j^2)q^2}{q-1}\sum_{\chi}
\begin{pmatrix}  D\chi \\ \chi \end{pmatrix}
\begin{pmatrix} D^2 \phi \chi^2 \\ D \chi \end{pmatrix}
\chi\left(\frac{-(1-j^2)^2}{16j^2}\right).
\]
By \cite[(2.15)]{TAMS} with $A=D \chi$, $B=\chi$, and $C=D^2 \phi \chi^2$,
\begin{equation*}
\begin{split}
&\frac{q^2}{q-1}
\begin{pmatrix}  D\chi \\ \chi \end{pmatrix}
\begin{pmatrix} D^2 \phi \chi^2 \\ D \chi \end{pmatrix} \\
&=\frac{q^2}{q-1}
\begin{pmatrix}  D^2 \phi \chi^2 \\ \chi \end{pmatrix}
\begin{pmatrix} D^2 \phi \chi \\ D \phi \chi \end{pmatrix}
-\chi(-1)\delta(D \chi) + D(-1) \delta(D^2 \phi \chi),
\end{split}
\end{equation*}
since by \cite[(2.6)]{TAMS},
\[
\begin{pmatrix} D^2 \phi \chi \\ D \end{pmatrix} =
\begin{pmatrix} D^2 \phi \chi \\ D \phi \chi \end{pmatrix}.
\]
Thus
\begin{equation}\label{eq:4.8}
\begin{split}
Q(D,j)=&\frac{\oD(16j^2)q^2}{q-1}\sum_{\chi}
\begin{pmatrix}  D^2 \phi \chi^2 \\ \chi \end{pmatrix}
\begin{pmatrix} D^2 \phi \chi \\ D \phi \chi \end{pmatrix}
\chi\left(\frac{-(1-j^2)^2}{16j^2}\right) \\
&-\oD^2(1-j^2) -D(-16j^2)\oD^4(1-j^2).
\end{split}
\end{equation}
By \cite[Theorem 4.16]{TAMS} with $A=D$, $B=\phi D^2$, and
$x = -(j+1)^2/(j-1)^2$, we have
\begin{equation*}
\begin{split}
&-\phi(j)D^4(j-1)\oD(16j^2)\frac{q}{q-1}\sum_{\chi}
\begin{pmatrix}  D^2 \phi \chi^2 \\ \chi \end{pmatrix}
\begin{pmatrix} D^2 \phi \chi \\ D \phi \chi \end{pmatrix}
\chi\left(\frac{-(1-j^2)^2}{16j^2}\right) = \\
&\2F1 \l( \bar D,D^2\phi \\ D\phi \ear \
-\left(\frac{j+1}{j-1}\right)^2  \r) -\phi(j)D(-16j^2)\oD^4(j+1)/q.
\end{split}
\end{equation*}
Multiply by $-\phi(j)q\oD^4(j-1)$ to get
\begin{equation*}
\begin{split}
&\oD(16j^2)\frac{q^2}{q-1}\sum_{\chi}
\begin{pmatrix}  D^2 \phi \chi^2 \\ \chi \end{pmatrix}
\begin{pmatrix} D^2 \phi \chi \\ D \phi \chi \end{pmatrix}
\chi\left(\frac{-(1-j^2)^2}{16j^2}\right) = \\
&-\phi(j)q\oD^4(j-1)\2F1 \l( \bar D,D^2\phi \\ D\phi \ear \
-\left(\frac{j+1}{j-1}\right)^2  \r) 
+D(-16j^2)\oD^4(1 -j^2).
\end{split}
\end{equation*}
Thus by \eqref{eq:4.8},
\begin{equation}\label{eq:4.9}
\begin{split}
& Q(D,j) + \oD^2(1-j^2) + D(-16j^2) \oD^4(1-j^2) = \\
&-\phi(j)q\oD^4(j-1)\2F1 \l( \bar D,D^2\phi \\ D\phi \ear \
-\left(\frac{j+1}{j-1}\right)^2  \r)
+D(-16j^2)\oD^4(1-j^2).
\end{split}
\end{equation}
Combining \eqref{eq:4.4} and \eqref{eq:4.9}, we arrive
at the desired result \eqref{eq:4.3}.
\end{proof}

\section{Double Mellin Transform  of $P(j,k)$}

For characters $\chi_1, \chi_2$, define the double Mellin transform
\begin{equation}\label{eq:5.1}
T=T(\chi_1, \chi_2): = \sum_{j, k \in \f} \chi_1(j)\chi_2(k)P(j,k).
\end{equation}
Note that $T(\chi_1, \chi_2)$ is symmetric in $\chi_1$, $\chi_2$.

The following theorem evaluates $T$.

\begin{thm}
If $\chi_1$ or $\chi_2$ is even, then $T=0$.
If $\chi_1$ and $\chi_2$ are both odd
(so that \eqref{eq:3.2} and \eqref{eq:3.3} hold), then
\begin{equation}\label{eq:5.2}
T=\sum_{i=0}^1 \omu \phi^i(a)\frac{G(\phi \mu^2)}{G(\phi)}
\{\sum_{j \in \f} \co(j) h(\mu \phi^i,j) +2(q-1)\delta(\mu \phi^i)\},
\end{equation}
where for a character $D$ on $\F$ and $j \in \f$, we define
\begin{equation}\label{eq:5.3}
h(D,j):=\sum_{x \in \f} D(x)\phi(1-x) \phi \oD^2(x(j+1)^2 +(j-1)^2).
\end{equation}
\end{thm}

\begin{proof}
By \eqref{eq:1.4}, $P(j, k)=-P(j,-k)$, so  $T=0$ if 
$\chi_1$ or $\chi_2$ is even.  Thus assume that
\eqref{eq:3.2} and \eqref{eq:3.3} hold.
Replacing $j$ by $jk$ in \eqref{eq:1.3}, we obtain
\begin{equation*}
\begin{split}
&T=2(q-1)\delta(\mu^4) + \\
&\frac{1}{G(\phi)}\sum_{x,j,k \in \f}
\co(j)\mu^4(k)\phi(a/x  - x)\psi(k^2(x(j+1)^2+a(j-1)^2/x)).
\end{split}
\end{equation*}
Since $\delta(\mu^4)=\delta(\mu^2)$, this becomes
\begin{equation*}
\begin{split}
&T=2(q-1)\delta(\mu^2) + \\
&\frac{1}{G(\phi)}\sum_{x,j,k \in \f}
\co(j)\mu^2(k)\phi(a/x  - x)\psi(k(x(j+1)^2+a(j-1)^2/x))(1+\phi(k)).
\end{split}
\end{equation*}
There is no contribution from the $1$ in the rightmost factor
$(1+\phi(k))$; to see this, replace $k$ and $x$ by their negatives.
Therefore,
\begin{equation*}
\begin{split}
&T=2(q-1)\delta(\mu^2) + \\
&\frac{G(\phi \mu^2)}{G(\phi)}\sum_{x,j \in \f}
\co(j)\phi(a/x  - x)\phi \omu^2(x(j+1)^2+a(j-1)^2/x).
\end{split}
\end{equation*}
It follows that
\begin{equation*}
\begin{split}
&T=2(q-1)\delta(\mu^2)\frac{G(\phi \mu^2)}{G(\phi)} + \\
&\frac{G(\phi \mu^2)}{G(\phi)}\sum_{x,j \in \f}
\co(j)\phi(a - x)\phi \omu^2(x(j+1)^2+a(j-1)^2)\mu(x)(1+\phi(x)).
\end{split}
\end{equation*}
After replacing $x$ by $ax$ and employing \eqref{eq:5.3},
the desired result \eqref{eq:5.2} readily follows.
\end{proof}

We proceed to  analyze $h(D,j)$.

\begin{lem}
We have
\begin{equation}\label{eq:5.4}
h(D,j) = -\phi(j) \oD(16) J(D, \phi), \quad \mbox{ if } j = \pm 1,
\end{equation}
and for $j \ne \pm 1$ and nontrivial $D$, we have
\begin{equation}\label{eq:5.5}
h(D,j) = \frac{G(\phi)G(D)^2}{G(\phi D^2)}\oD^4(j-1)
\2F1 \l( \bar D,D^2\phi \\ D\phi \ear \
-\left(\frac{j+1}{j-1}\right)^2  \r).
\end{equation}
Finally, if $j \ne \pm 1$ and $D$ is trivial, then $h(D,j)=0$.
\end{lem}

\begin{proof}
The evaluation in \eqref{eq:5.4} follows directly from the definition
of $h(D,j)$ in \eqref{eq:5.3}.  The evaluation in \eqref{eq:5.5}
is the same as that in \cite[(5.21)]{Evans}, the proof of which is 
valid for $q$ congruent to either 1 or 3 mod 4.  Finally, let
$j \ne \pm 1$.  Then since
\[
h(\e,j) = -1 + \sum_{x \in \F} \phi(1-x) \phi(x(j+1)^2 + (j-1)^2),
\]
replacement of $x$ by $1-x(2j^2+2)/(j+1)^2$ shows that 
$h(\e,j)= -1 + 1 = 0$.
\end{proof}

\begin{thm}
For a character $D$ on $\F$, define
\begin{equation}\label{eq:5.6}
W(D): = \sum_{j \in \f} \phi \no^4(j)h(D,j).
\end{equation}
Then $W(\e)=2$, and for nontrivial $D$,
\begin{equation}\label{eq:5.7}
W(D)=\frac{-G(\phi)G(D)^2}{qG(\phi D^2)}
\{J_2(\no N M_8, \oD N)+ J_2(\no N M_8^5, \oD N)\}.
\end{equation}
\end{thm}

\begin{proof}
It follows directly from Lemma 5.2 that $W(\e)=2$.
Let $D$ be nontrivial.   By Lemma 5.2,
\begin{equation}\label{eq:5.8}
\begin{split}
&W(D)=-2\oD(16)J(D, \phi) \ - \\
&\frac{G(\phi)G(D)^2}{qG(\phi D^2)}
\sum_{\substack{j \in \f \\ j \ne \pm 1}}
\no^4(j)\left(-q\phi(j)\oD^4(j-1)
\2F1 \l( \bar D,D^2\phi \\ D\phi \ear \
-\left(\frac{j+1}{j-1}\right)^2  \r)\right).
\end{split}
\end{equation}
Thus by \eqref{eq:4.1}--\eqref{eq:4.3},
\begin{equation*}
\begin{split}
&W(D) = -2\oD(16)J(D, \phi) -\frac{G(\phi)G(D)^2}{qG(\phi D^2)}
\sum_{\substack{j \in \f \\ j \ne \pm 1}}
\no^4(j)R(D,j) = \\
&-2\oD(16)J(D, \phi) -
\frac{G(\phi)G(D)^2}{qG(\phi D^2)}\{2\oD(4)J(\phi D^2, \phi) \  +
\sum_{j \in \f } \no^4(j)R(D,j)\}.
\end{split}
\end{equation*}
This simplifies to
\begin{equation}\label{eq:5.9}
W(D) = -\frac{G(\phi)G(D)^2}{qG(\phi D^2)}\sum_{j \in \f } \no^4(j)R(D,j).
\end{equation}
For brevity, let $Y(D)$ denote this sum on $j$.
It remains to prove that
\begin{equation}\label{eq:5.10}
Y(D): = \sum_{j \in \f } \no^4(j)R(D,j)=
J_2(\no N M_8, \oD N)+ J_2(\no N M_8^5, \oD N).
\end{equation}
Since the fourth powers in $\F$ are precisely the squares,
it follows from definition \eqref{eq:4.1} that
\[
Y(D) = \sum_{j \in \f } \no(j^2) 
\sum_{\substack{z \in \FF \\ N(z)=j^2}}
M_8(z)\oD N(1-z).
\]
Thus
\begin{equation*}
\begin{split}
Y(D)=& \frac{1}{q-1}\sum_{j \in \f } \no(j^2) 
\sum_{z \in \ff}M_8(z)\oD N(1-z)\sum_{\chi} \chi(N(z)/j^2)= \\
&\frac{1}{q-1}\sum_{\chi} J_2(\chi N M_8, \oD N) 
\sum_{j \in \f} (\no \oc)^2(j).
\end{split}
\end{equation*}
The sum on $j$ on the right vanishes unless $\chi \in \{\no, \no \phi\}$,
and so we obtain the desired result \eqref{eq:5.10}.
\end{proof}

As interesting consequences of Theorem 5.3, we record the elegant 
double character sum evaluations \eqref{eq:5.11}--\eqref{eq:5.14} below.

\begin{thm}
For any character $\nu$ on $\F$,
\begin{equation}\label{eq:5.11}
\begin{split}
\sum_{j, x \in \f} &\phi \nu^4(j) \phi(x)\phi(1-x)\phi(x(j+1)^2+(j-1)^2) = \\
&J_2(\nu N M_8, \phi N) + J_2(\nu N M_8^5, \phi N).
\end{split}
\end{equation}
\end{thm}

\begin{proof}
This follows by putting $D=\phi$ in \eqref{eq:5.7}.
\end{proof}

\begin{thm}
When $q \equiv 7 \pmod 8$, we have
\begin{equation}\label{eq:5.12}
\sum_{j, x \in \f} \phi(jx)\phi(1-x)\phi(x(j+1)^2+(j-1)^2)=2q.
\end{equation}
When $q \equiv 3 \pmod 8$, we have
\begin{equation}\label{eq:5.13}
\sum_{j, x \in \f} \phi(jx)\phi(1-x)\phi(x(j+1)^2+(j-1)^2)=2u,
\end{equation}
where $|u|$, $|v|$ is the 
unique pair of positive integers with 
$p \nmid u$ for which $q^2=u^2 + 2v^2$,
and where the sign of $u$ is determined by the congruence
$u \equiv -1 \pmod 8$.
In particular, when $q=p \equiv 3 \pmod 8$, we have
\begin{equation}\label{eq:5.14}
\sum_{j, x \in \f} \phi(jx)\phi(1-x)\phi(x(j+1)^2+(j-1)^2)=4a_8^2 -2p,
\end{equation}
where $p = a_8^2 + 2b_8^2$.
\end{thm}

\begin{proof}
By \eqref{eq:5.11} with $\nu = \e$, the sum in \eqref{eq:5.12} equals
\[
J_2(M_8, \phi N) + J_2(M_8^5, \phi N).
\]
First suppose that $q \equiv 7 \pmod 8$.  Then
\[
G_2(M_8) = G_2(M_8^5), \quad
G_2(\phi N) = q
\]
by \cite[Theorem 11.6.1]{BEW}.
Thus each Jacobi sum above equals $q$, which proves
\eqref{eq:5.12}.

Now suppose that $q \equiv 3 \pmod 8$.
An application of \eqref{eq:2.1}
shows that $J_2(M_8^5, \phi N)$ is the complex
conjugate of $J_2(M_8, \phi N)$, 
so that the sum in \eqref{eq:5.13} equals $2 \Re J_2(M_8, \phi N)$.

First consider the case where $q$ is prime, i.e., $q=p$.
Then $G_2(\phi N) = p$ and
by \cite[Theorems 12.1.1 and 12.7.1(b)]{BEW}, 
\[
G_2(M_8) = G(\phi)\pi, \quad G_2(M_8^5)=G(\phi)\opi, \quad
J_2(M_8, \phi N) = \pi^2,
\]
where $\pi=a_8 + i b_8 \sqrt{2}$ is a prime in $\Q(i\sqrt{2})$
of norm $p=\pi \opi=a_8^2 + 2b_8^2$.  
Note that
$\pi^2 = u_1 + iv_1\sqrt{2}$,
where
\[
u_1 = 2a_8^2 - p, \quad v_1= 2a_8b_8, \quad u_1^2 + 2v_1^2 = p^2,
\]
so that
\[
\Re J_2(M_8, \phi N) = u_1 = 2a_8^2 -p \equiv -1 \pmod 8.
\]

In the general case where say $q = p^t$, the Hasse-Davenport
lifting theorem \cite[Theorem 11.5.2]{BEW} yields
\[
J_2(M_8, \phi N) = (-1)^{t-1}\pi^{2t}=
(-1)^{t-1}(u_1 +iv_1\sqrt{2})^t=u+iv\sqrt{2},
\]
for integers $u$, $v$ such that $q^2 = u^2 + 2v^2$.   
Since $u_1 \equiv -1 \pmod 8$, it is easily seen using the binomial
theorem that $u \equiv -1 \pmod 8$. If $p=\pi\opi$ divided $u$,
then $p$ would divide $v$, so that the prime  $\opi$ would divide $\pi^{2t}$,
which is impossible.  Thus $p \nmid u$.
For an elementary proof of
the uniqueness of $|u|$, $|v|$, see \cite[Lemma 3.0.1]{BEW}.
\end{proof}

\noindent
\textit{Remark:}
The sum in Theorem 5.5, namely
\begin{equation}\label{eq:5.15}
Z =\sum_{j \in \f} \phi(j) h(\phi, j),
\end{equation}
can be evaluated when $q \equiv 1 \pmod 4$ as well.
We have $Z=0$ when $q \equiv 5 \pmod 8$,
which can be seen by applying \cite[Lemma 5.1]{Evans}
with $\phi$ in place of $D$, and then replacing $j$ by $jI$,
where $I$ is a primitive fourth root of
unity in $\F$. More work is needed to evaluate $Z$ in the remaining case
where  $q \equiv 1 \pmod 8$.  In this case $Z$ is equal to the
sum $R_2$ in \cite[(5.44)]{Evans} with $\no = B_8$ and
$A_4 = B_8^2$ for an octic character $B_8$ on $\F$.
The proof of \cite[Theorem 3.3.1]{BEW} shows that
\[
J(B_8, \phi) = J(B_8^3, \phi) \in \Q(i\sqrt{2}).
\]
Using this equality to evaluate the sum $R_2$, we have
\begin{equation}\label{eq:5.16}
Z= 2q + 2 \Re J(B_8, \phi)^2, \quad q \equiv 1 \pmod 8.
\end{equation}
We will use \eqref{eq:5.16} to show that
\begin{equation}\label{eq:5.17}
Z=4q, \ \mbox{ when } p \equiv 5 \mbox{ or } 7 \pmod 8,
\end{equation}
and
\begin{equation}\label{eq:5.18} 
Z=4c^2, \ \mbox{ when } p \equiv 1 \mbox{ or } 3 \pmod 8,
\end{equation}
where $c$ and $d$ are the unique pair of integers up to sign
for which
\begin{equation}\label{eq:5.19}
q = c^2 + 2d^2, \quad  p \nmid c.
\end{equation}

First suppose that $p \equiv 7 \pmod 8$.  Then $q = p^{2t}$
for some $t \ge 1$.  If $t=1$, then $J(B_8, \phi)=p$
by \cite[Theorem 11.6.1]{BEW}.   For general $t$, the 
Hasse-Davenport lifting theorem thus yields
$J(B_8, \phi)=(-1)^{t-1}p^t$, so that $J(B_8, \phi)^2=q$.
Thus $Z=4q$ by \eqref{eq:5.16}.

Now suppose that $p \equiv 5 \pmod 8$. 
Since $G(B_8)=G(B_8^p)=G(B_8^5)$ by \cite[Theorem 1.1.4(d)]{BEW},
$J(B_8, \phi) = G(\phi)$.   Thus
$J(B_8, \phi)^2=q$, so again $Z=4q$.   This completes the proof of
\eqref{eq:5.17}.

Next suppose that $p \equiv 3 \pmod 8$.
Then $q = p^{2t}$ for some $t \ge 1$. 
Since $-2$ is a square $\pmod p$, we have the prime splitting
$p = \pi \opi$ in $\Q(i\sqrt{2})$.
Assume first that $t=1$.   Then
\begin{equation}\label{eq:5.20}
J(B_8, \phi) \oJ (B_8, \phi) = q = p^2 = \pi^2 \opi^2, \quad t=1.
\end{equation}
We cannot have $J(B_8, \phi)= \pm p$, otherwise the 
prime ideal factorization
of $J(B_8, \phi)$ in \cite[Theorems 11.2.3, 11.2.9]{BEW} would yield the
contradiction that $p$ ramifies in the cyclotomic field $\Q(\exp(2\pi i/8))$.
In view of \eqref{eq:5.20} and unique factorization in $\Q(i\sqrt{2})$,
we may suppose without loss of generality
that $J(B_8, \phi) = \pi^2$ when $t=1$.   For general $t$,
\begin{equation}\label{eq:5.21}
J(B_8, \phi) = (-1)^{t-1} \pi^{2t} = c+ di\sqrt{2}
\end{equation}
for some integers $c$ and $d$ such $c^2 + 2d^2 = q$.
Note that $p$ cannot divide $c$, for otherwise
$p$ also divides $d$ (since $c^2 + 2d^2 = q$), 
so that $p$ divides $\pi^{2t}$, yielding the
contradiction that the prime $\opi$ divides $\pi$.  By \eqref{eq:5.21},
\[
\Re J(B_8, \phi)^2 = c^2 -2d^2 = 2c^2 -q,
\]
so that by \eqref{eq:5.16},  $Z = 4c^2$.

Finally, suppose that $p \equiv 1 \pmod 8$, and write $q=p^t$
for some $t \ge 1$. 
Since $-2$ is a square $\pmod p$, we have the prime splitting
$p = \pi \opi$ in $\Q(i\sqrt{2})$.
If $t=1$, then without loss of generality, 
$J(B_8, \phi) = \pi$.   For general $t$, 
\[
J(B_8, \phi) = (-1)^{t-1} \pi^t = c+ di\sqrt{2}
\]
for some integers $c$ and $d$ such $c^2 + 2d^2 = q$.
Arguing as in the case
$p \equiv 3 \pmod 8$,
we again obtain $Z=4c^2$ for $c$ as in \eqref{eq:5.19}.
This completes the proof of \eqref{eq:5.18}.

\section{Proof of Katz's identities (1.5)}

When $jk=0$, both sides of \eqref{eq:1.5} vanish,
by \eqref{eq:1.4} and \eqref{eq:1.8}.
We thus assume that $jk \ne 0$.
It suffices to show that 
the Mellin transforms of the left
and right sides of \eqref{eq:1.5} are the same for all characters,
for then  \eqref{eq:1.5} follows by taking inverse Mellin transforms.
Thus it remains to show that $S=T$,
where $S$ and $T$ are given in Theorems 3.1 and 5.1, respectively.
These theorems show that $S$ and $T$ both vanish when $\chi_1$
or $\chi_2$ is even, so we may assume that
\eqref{eq:3.2} and \eqref{eq:3.3} hold.
For brevity, write $D = \mu \phi^i$, where $i \in \{0,1\}$.
Then the equality $S=T$ is equivalent to
\begin{equation}\label{eq:6.1}
\begin{split}
&\frac{q }{G_2(\oD N)}
\{J_2(\nu_1 N M_8, \oD N) + J_2(\nu_1 N M_8^5, \oD N)\} = \\
&\frac{G(\phi D^2)}{G(\phi)}(W(D) + 2(q-1)\delta(D)).
\end{split}
\end{equation}
Noting that $G_2(\oD N)= -G(\oD)^2$ by \eqref{eq:2.2}, and using the formula
for $W(D)$ in Theorem 5.3, we easily see that
\eqref{eq:6.1} holds.  This completes the proof that $S=T$.

\end{document}